\UseRawInputEncoding
\documentclass[11pt, reqno]{amsart}
\usepackage{amsmath, amsthm}
\usepackage{amscd}
\usepackage{enumerate}
\usepackage{float,amsmath,amssymb,mathrsfs,bm,multirow,graphics}
\usepackage[dvips]{graphicx}
\usepackage[percent]{overpic}
\usepackage[numbers,sort&compress]{natbib}
\usepackage{xcolor}
\usepackage{todonotes}
\usepackage{amsaddr}

\allowdisplaybreaks

\newcommand{\R}{{\mathbb{R}}}

\newcommand{\C}{{\mathbb{C}}}
\newcommand{\Z}{{\mathbb{Z}}}

\newcommand{\res}{\text{\upshape Res} \,}

\newcommand{\cut}{\text{cut}}

\newtheorem{theorem}{Theorem}[section]

\newtheorem{remark}[theorem]{Remark}

\newtheorem{RHproblem}[theorem]{RH problem}

\numberwithin{equation}{section}
\usepackage[colorlinks=true]{hyperref}
\hypersetup{urlcolor=blue, citecolor=red, linkcolor=blue}

\begin{document}

\title[The Implementation of the Unified Transform]{The Implementation of the Unified Transform \\ to the Nonlinear Schr\"odinger Equation with \\ Periodic Initial Conditions}

\author{B. Deconinck}
\address{Department of Applied Mathematics, University of Washington, \\ Seattle, WA 98195-3925, United States.}
\email{deconinc@uw.edu}

\author{A. S. Fokas}
\address{Department of Applied Mathematics and Theoretical Physics, \\ University of Cambridge, Cambridge CB3 0WA, United Kingdom.}
\email{t.fokas@damtp.cam.ac.uk} 

\author{J. Lenells}
\address{Department of Mathematics, KTH Royal Institute of Technology, \\ 100 44 Stockholm, Sweden.}
\email{jlenells@kth.se}

\begin{abstract} 
\noindent
The unified transform method (UTM) provides a novel approach to the analysis of initial-boundary value problems for linear as well as for a particular class of nonlinear partial differential equations called integrable. If the latter equations are formulated in two dimensions (either one space and one time, or two space dimensions), the UTM expresses the solution in terms of a matrix Riemann-Hilbert (RH) problem with explicit dependence on the independent variables. For nonlinear integrable evolution equations, such as the celebrated nonlinear Schr\"odinger (NLS) equation, the associated jump matrices are computed in terms of the initial conditions and all boundary values. The unknown boundary values are characterized in terms of the initial datum and the given boundary conditions via the analysis of the so-called global relation. In general, this analysis involves the solution of certain nonlinear equations. In certain cases, called linearizable, it is possible to bypass this nonlinear step. In these cases, the UTM solves the given initial-boundary value problem with the same level of efficiency as the well-known inverse scattering transform solves the initial value problem on the infinite line. 
We show here that the initial-boundary value problem on a finite interval with $x$-periodic boundary conditions (which can alternatively be viewed as the initial value problem on a circle), belongs to the linearizable class. Indeed, by employing certain transformations of the associated RH problem and by using the global relation, the relevant jump matrices can be expressed explicitly in terms of the so-called scattering data, which are computed in terms of the initial datum. Details are given for NLS, but similar considerations are valid for other well-known integrable evolution equations, including the Korteweg-de Vries (KdV) and modified KdV equations. 
\end{abstract}

\maketitle

\section{Introduction}
Two of the most extensively studied classes of solutions of nonlinear partial differential equations of evolution type are (i) solutions with decay at spatial infinity and (ii) solutions which are periodic in space. For {\it integrable} evolution equations, the introduction of the Inverse Scattering Transform (IST) about half a century ago had the groundbreaking implication that the initial value problem (IVP) for solutions in the first of these classes could be solved by means of only {\it linear} operations \cite{GGKM1967}. More precisely, this approach expresses the solution of the IVP problem in terms of the solution of a linear singular integral equation, or, equivalently, in terms of the solution of a Riemann-Hilbert (RH) problem. 
Since this RH problem is formulated in terms of spectral functions whose definition only involves the given initial data, the solution formula is effective.

For the second of the above classes---the class of spatially periodic solutions---the introduction of the so-called finite-gap integration method in the 1970s had far-reaching implications, see \cite{BBEIM1994}. This approach makes it possible to generate large classes of exact solutions, the so-called finite-gap solutions, by considering rational combinations of theta functions defined on Riemann surfaces with a finite number of branch cuts. The construction of solutions via this approach, in particular of the Korteweg-de Vries (KdV) and the nonlinear Schr\"odinger (NLS) equations, has a long and illustrious history \cite{M2008}. The theory  combines several branches of mathematics, including the spectral theory of differential operators with periodic coefficients, Riemann surface theory, and algebraic geometry, and has had implications for diverse areas of mathematics \cite{BBEIM1994, M2008, GH2003}.  

Despite the unquestionable success of the finite-gap approach, two questions naturally present themselves: 
\begin{enumerate}[$-$]

\item How effective is the finite-gap approach for the solution of the IVP for general periodic initial data? If the initial data are in the finite-gap class, the solution of the IVP can be retrieved by solving a Jacobi inversion problem on a finite-genus Riemann surface. However, if the initial data are not in the class of finite-gap potentials, then a theory of infinite-gap solutions is required. Although some aspects of the approach certainly can be extended to Riemann surfaces of infinite-genus \cite{FKT2003, MvM1975, MSS1998}, there are several complications associated with solving the Jacobi inversion problem on an infinite-genus surface; in fact, from a computational point of view even finite-gap solutions are not easily accessible \cite{BK2011, BBEIM1994}.

\item Why is the finite-gap integration method conceptually so different from the IST formalism? For example, whereas the IST relies on the solution of a RH problem, the algebro-geometric approach calls for the solution of a Jacobi inversion problem.
(A Jacobi inversion problem can be regarded as a RH problem on a branched plane \cite{MN2019}, but conceptually such an approach appears indirect.)
\end{enumerate}

We propose a new approach for the solution of the space-periodic IVP for an integrable evolution equation. Just like in the implementation of the IST for the problem on the line, 
our approach expresses the solution of the periodic problem in terms of the solution of a RH problem, and the definition of this RH problem involves only the given initial data. Thus the presented solution of the periodic problem is conceptually analogous to the solution of the problem on the line. 

Let us explain the ideas that led us to the presented approach. Although the space-periodic problem is often viewed as an IVP on the circle, it can also be viewed as an initial-boundary value problem (IBVP) on an interval with periodic boundary conditions. In 1997, one of the authors introduced a methodology for the solution of IBVPs for integrable equations \cite{F1997}.
The implementation of this method, known as the Unified Transform or Fokas Method, to an IBVP typically consists of two steps. The first step is to construct an integral representation of the solution characterized by a RH problem formulated in the complex $k$-plane, where $k$ denotes the spectral parameter of the associated Lax pair. Since this representation involves, in general, some unknown boundary values, the solution formula is not yet effective. The second step is therefore to characterize the unknown boundary values by analyzing a certain equation called the {\it global relation}. In general, the second step involves the solution of a nonlinear problem. However, for certain so-called {\it linearizable} boundary conditions, additional symmetries can be used to eliminate the unknown boundary values from the formulation of the RH problem altogether. 

It turns out that the boundary conditions corresponding to the space-periodic problem are linearizable. This suggests that it should be possible to use the above philosophy to obtain an effective solution. Although the analysis is quite different from that of other linearizable problems, we have indeed been able to find a RH problem whose formulation only involves the initial data. 

The new approach is as effective for the periodic problem as the IST is for the problem on the line in the sense that all ingredients of the RH problem are defined in terms of quantities obtained by solving a linear Volterra integral equation whose kernel is characterized by the initial data. 
However, the formulation of the RH problem in the periodic case is complicated by the fact that the jump contour is determined by the zeros of an entire function (more precisely, the jump contour involves the spectral gaps known from the finite-gap approach). 

Section \ref{section2} presents the main result. Section \ref{constantexamplesec} presents details for the example when $q_0(x) = q_0$ is a constant, which provides an illustration of the effectiveness of the new formalism, as well as a verification of its correctness.

\section{The Main Result}\label{section2}
We write $D_1, \dots D_4$ for the four open quadrants of the complex plane, see Figures \ref{contourfig} and \ref{jumps.pdf}. A bar over a function  denotes that the Schwarz conjugate is taken, i.e., $\bar{f}(k) = \overline{f(\bar{k})}$. We define $\theta$ by
\begin{align}\label{thetadef}
\theta = \theta(x,t,k) = kx+2k^2 t. 
\end{align}

\begin{theorem}\label{mainth}
Suppose that $q(x,t)$ is a smooth solution of the NLS equation
\begin{align}\label{NLS}
i q_t+q_{xx}-2\lambda q |q|^2=0, \qquad \lambda = \pm 1,
\end{align}
which is $x$-periodic of period $L > 0$, i.e., $q(x+L, t) = q(x,t)$.
Define $a(k)$ and $b(k)$ in terms of the initial datum 
\begin{align}
q_0(x) = q(x,0), \qquad 0 < x < L,
\end{align}
by 
\begin{align}
 a(k) = \mu^{(2)}(0,k), \qquad b(k) = \mu^{(1)}(0,k), \qquad k \in \C,
\end{align}
where the vector $\big(\mu^{(1)}(x,k), \mu^{(2)}(x,k)\big)$ satisfies the linear Volterra integral equation
\begin{align}
\begin{cases} \displaystyle{ \mu^{(1)}(x,k) = - \int_x^L e^{2ik (y-x)} q_0(y) \mu^{(2)}(y,k) dy,}
  \vspace{.05cm}	\\
\displaystyle{ \mu^{(2)}(x,k) = 1 - \lambda \int_x^L \overline{q_0(y)} \mu^{(1)}(y,k) dy, }
 \end{cases}
 \qquad 0 < x < L, \;\; k \in \C.
\end{align}
Define $\tilde{\Gamma}(k)$  in terms of $a(k)$  and $b(k)$  by
\begin{align}\label{tildeGammadef}
\tilde{\Gamma} = \frac{\lambda}{2 e^{ikL} \bar{b}} \Big(\bar{a} e^{ikL} - a e^{-ikL} - i\sqrt{4 - \Delta^2}\Big), 
\end{align}
where $\Delta(k)$ is given by 
\begin{align}\label{Deltadef}
  \Delta = a e^{-ikL}+\bar{a} e^{ikL}, \qquad k \in \C.
\end{align}

Define the contour $\tilde{\Sigma}$ by 
$\tilde{\Sigma} = \R \cup i\R \cup \mathcal{C}$,
where $\mathcal{C}$ denotes a set of branch cuts of $\sqrt{4 - \Delta^2}$ consisting of subintervals of $\R$ and vertical line segments which connect the zeros of odd order of $4 - \Delta^2$; orient $\mathcal{C}$ as in Figure \ref{contourfig} and let $\tilde{\Sigma}_\star$ denote the contour $\tilde{\Sigma}$ with all branch points of $\sqrt{4 - \Delta^2}$ and points of self-intersection removed, see \cite{FLxperiodic} for details.

Assume that the function $\tilde{\Gamma}$ has no poles on the contour $\tilde{\Sigma}$ and that it has at most finitely many simple poles in $D_1 \cup D_3$; denote the possible poles in $D_1$  by $\{p_j\}_1^{n_1}$ and the possible poles in $D_3$ by $\{q_j\}_{1}^{n_2}$. Let $\tilde{P} = \{p_j, \bar{p}_j\}_1^{n_1} \cup \{q_j, \bar{q}_j\}_1^{n_2}$.

Consider the following RH problem: Find a $2 \times 2$-matrix-valued function $\tilde{m}$ such that
\begin{itemize}
\item $\tilde{m}(x,t,\cdot) : \C \setminus (\tilde{\Sigma} \cup \tilde{P}) \to \C^{2 \times 2}$ is analytic.

\item The limits of $m(x,t,k)$ as $k$ approaches $\tilde{\Sigma}_\star$ from the left $(+)$ and right $(-)$ exist, are continuous on $\tilde{\Sigma}_\star$, and satisfy
\begin{align}\label{mtildejump}
  \tilde{m}_-(x,t,k) = \tilde{m}_+(x, t, k) \tilde{v}(x, t, k), \qquad k \in \tilde{\Sigma}_\star,
\end{align}
where $\tilde{v}$ is defined by 
\begin{align}\label{tildevdef}
\tilde{v} = \begin{cases}  \tilde{v}_1 = \begin{pmatrix}  \frac{a  - \lambda b \bar{\tilde{\Gamma}} - \lambda \tilde{\Gamma} (\bar{a}\bar{\tilde{\Gamma}} - \bar{b})e^{2ikL} }{a- \lambda b \bar{\tilde{\Gamma}}} & -\tilde{\Gamma}  e^{2ikL} e^{-2 i \theta } \\
 \frac{\lambda \bar{\tilde{\Gamma}} e^{2 i \theta } }{(a - \lambda b \bar{\tilde{\Gamma}}) (a + \lambda\bar{b} \tilde{\Gamma}  e^{2ikL})} & \frac{a}{a + \lambda \bar{b} \tilde{\Gamma}  e^{2ikL}} \end{pmatrix}, & k \in i\R_+ \setminus \bar{\mathcal{C}},
	\\ 
\tilde{v}_2 =  \begin{pmatrix} 1 - \lambda \tilde{\Gamma} \bar{\tilde{\Gamma}} 
& -\frac{(\bar{a} \tilde{\Gamma} e^{2ikL}+b) e^{-2 i \theta }}{\bar{a}+ \lambda b \bar{\tilde{\Gamma}} e^{-2ikL}} \\
 \frac{\lambda (a \bar{\tilde{\Gamma}}e^{-2ikL}+\bar{b}) e^{2 i \theta } }{a+ \lambda \bar{b} \tilde{\Gamma}  e^{2ikL} } &
   \frac{1}{(a+ \lambda \bar{b} \tilde{\Gamma}  e^{2ikL}) (\bar{a} + \lambda b \bar{\tilde{\Gamma}} e^{-2ikL})}  \end{pmatrix}, & k \in \R_+ \setminus \bar{\mathcal{C}},
   	\\  
\tilde{v}_3 = \begin{pmatrix} \frac{\bar{a}  - \lambda \bar{b} \tilde{\Gamma} - \lambda \bar{\tilde{\Gamma}}(a\tilde{\Gamma} - b) e^{-2ikL}}{\bar{a} - \lambda \bar{b} \tilde{\Gamma} } & - \frac{\tilde{\Gamma} e^{-2 i \theta}}{(\bar{a} - \lambda \bar{b} \tilde{\Gamma}) (\bar{a} + \lambda b \bar{\tilde{\Gamma}} e^{-2ikL} )} \\
 \lambda \bar{\tilde{\Gamma}} e^{-2ikL} e^{2 i \theta} & \frac{\bar{a}}{\bar{a} + \lambda b \bar{\tilde{\Gamma}} e^{-2ikL}}  \end{pmatrix}, & k \in i\R_- \setminus \bar{\mathcal{C}},
   	\\ 
 \tilde{v}_4 = \begin{pmatrix}	
 \frac{1 - \lambda \tilde{\Gamma} \bar{\tilde{\Gamma}}}{(a - \lambda b \bar{\tilde{\Gamma}}) (\bar{a} - \lambda \bar{b} \tilde{\Gamma} )} & -\frac{(a \tilde{\Gamma} - b)e^{-2 i \theta } }{\bar{a} - \lambda \bar{b} \tilde{\Gamma} } \\
 \frac{\lambda(\bar{a} \bar{\tilde{\Gamma}} - \bar{b})e^{2 i \theta }}{a - \lambda b \bar{\tilde{\Gamma}}} & 1
 \end{pmatrix}, & k \in \R_- \setminus \bar{\mathcal{C}},
 	\\
\tilde{v}^{\cut}_{D_1} = 
\begin{pmatrix} \frac{a + \lambda \bar{b} \tilde{\Gamma}_+ e^{2ikL}}{a + \lambda \bar{b} \tilde{\Gamma}_- e^{2ikL}} & (\tilde{\Gamma}_- - \tilde{\Gamma}_+) e^{2ikL} e^{-2i\theta} \\ 0 & \frac{a + \lambda \bar{b} \tilde{\Gamma}_- e^{2ikL}}{a + \lambda \bar{b}\tilde{\Gamma}_+ e^{2ikL}} \end{pmatrix}, & k \in \mathcal{C} \cap D_1,
	\\
\tilde{v}^{\cut}_{D_2} = \begin{pmatrix} 1 & 0 \\ 
\frac{\lambda (\bar{\tilde{\Gamma}}_- - \bar{\tilde{\Gamma}}_+)e^{2i\theta}}{(a - \lambda b \bar{\tilde{\Gamma}}_-)(a - \lambda b \bar{\tilde{\Gamma}}_+)} & 1 \end{pmatrix},
& k \in \mathcal{C} \cap D_2,
	\\
\tilde{v}^{\cut}_{D_3} = \begin{pmatrix} 1 & \frac{(\tilde{\Gamma}_- - \tilde{\Gamma}_+)e^{-2i\theta}}{(\bar{a} - \lambda \bar{b} \tilde{\Gamma}_-)(\bar{a} - \lambda \bar{b} \tilde{\Gamma}_+)} \\
0 & 1 \end{pmatrix}, & k \in \mathcal{C} \cap D_3,
	\\
\tilde{v}^{\cut}_{D_4} = \begin{pmatrix} \frac{\bar{a} + \lambda b \bar{\tilde{\Gamma}}_- e^{-2ikL}}{\bar{a} + \lambda b \bar{\tilde{\Gamma}}_+ e^{-2ikL}} & 0 \\
\lambda(\bar{\tilde{\Gamma}}_- - \bar{\tilde{\Gamma}}_+) e^{2i\theta} e^{-2ikL} & \frac{\bar{a} + \lambda b \bar{\tilde{\Gamma}}_+ e^{-2ikL}}{\bar{a} + \lambda b \bar{\tilde{\Gamma}}_- e^{-2ikL}} \end{pmatrix}, & k \in \mathcal{C} \cap D_4,
	\\
\tilde{v}_1^{\cut} = \tilde{v}^{\cut}_{D_1} \tilde{v}_{1-},
& k \in \mathcal{C} \cap i\R_+,	
	\\
\tilde{v}_2^{\cut} = \tilde{v}^{\cut}_{D_1} \tilde{v}_{2-},
& k \in \mathcal{C} \cap \R_+,	
	\\
\tilde{v}_3^{\cut} = \tilde{v}^{\cut}_{D_3} \tilde{v}_{3-},
& k \in \mathcal{C} \cap i\R_-,	
	\\
\tilde{v}_4^{\cut} = \tilde{v}^{\cut}_{D_3} \tilde{v}_{4-},
& k \in \mathcal{C} \cap \R_-,
\end{cases}
\end{align}
and $\tilde{v}_{j-}$ denotes the boundary values of $\tilde{v}_j$ as $k$ approaches $\mathcal{C}$ from the right.

\item $\tilde{m}(x,t,k) = I  + O\big(k^{-1}\big)$ as $k \to \infty$, $k \in \C \setminus \cup_{n\in \Z} \mathcal{D}_n$, where $\mathcal{D}_n$ is the open disk 
\begin{align}\label{Dndef}
\mathcal{D}_n = \bigg\{k \in \C \, \bigg| \, \bigg|k - \frac{n\pi}{L}\bigg| < \frac{\pi}{4L}\bigg\}.
\end{align}

\item $\tilde{m}(x,t,k) = O(1)$ as $k \to \tilde{\Sigma} \setminus \tilde{\Sigma}_\star$, $k \in \C \setminus \tilde{\Sigma}$.

\item At the points $p_j \in D_1$  and  $\bar{p}_j \in D_4$, $\tilde{m}$ satisfies, for $j = 1, \dots, n_1$,
\begin{subequations}\label{pqresidues}
\begin{align}
& \underset{k = p_j}{\res} [\tilde{m}(x,t,k)]_2 =  \Big[[\tilde{m}]_1 (a + \lambda \bar{b} \tilde{\Gamma} e^{2ikL})\bar{a} e^{2ikL} e^{-2i\theta} \Big](x,t,p_j) \, \underset{k=p_j}{\res} \tilde{\Gamma}(k),
	\\
& \underset{k=\bar{p}_j}{\res} [\tilde{m}(x,t,k)]_1 = \Big[\lambda [\tilde{m}]_2 (\bar{a} + \lambda b \bar{\tilde{\Gamma}} e^{-2ikL}) a e^{-2ikL} e^{2i\theta} \Big](x,t,\bar{p}_j) \, \overline{\underset{k=p_j}{\res} \tilde{\Gamma}(k)}.
\end{align}

\item At the points $q_j \in D_3$  and  $\bar{q}_j \in D_2$, $\tilde{m}$ satisfies, for $j = 1, \dots, n_2$,
\begin{align}
& \underset{k=q_j}{\res} [\tilde{m}(x,t,k)]_2 = \bigg[ [\tilde{m}]_1 \frac{a e^{-2i\theta}}{\bar{a} - \lambda \bar{b} \tilde{\Gamma}}\bigg](x,t,q_j) \, \underset{k=q_j}{\res} \tilde{\Gamma}(k),
	\\
& \underset{k=\bar{q}_j}{\res} [\tilde{m}(x,t,k)]_1 = \bigg[\lambda [\tilde{m}]_2 \frac{\bar{a} e^{2i\theta}}{a - \lambda b \bar{\tilde{\Gamma}}}\bigg](x,t,\bar{q}_j) \, \overline{\underset{k=q_j}{\res} \tilde{\Gamma}(k)}.
\end{align}
\end{subequations}

\end{itemize}

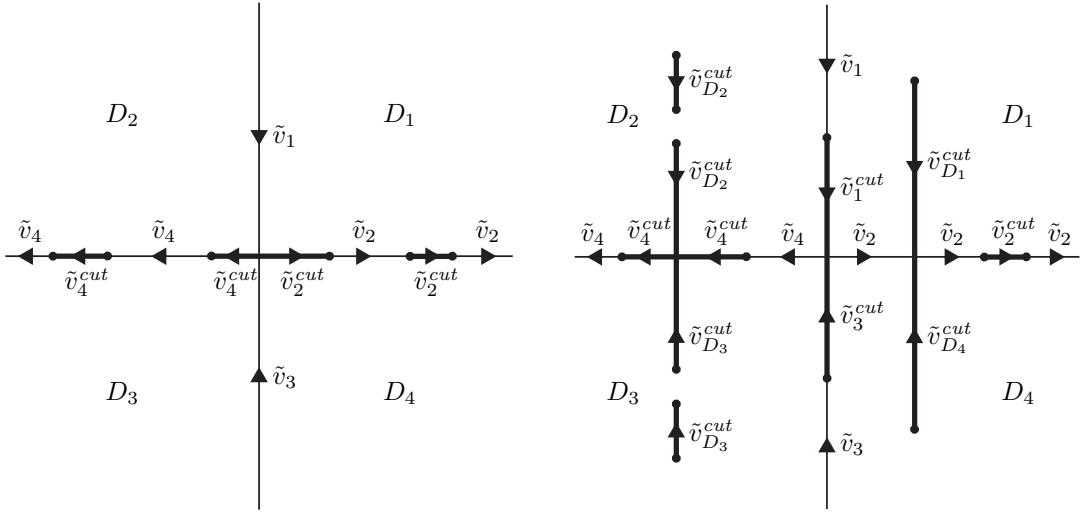
\begin{figure}
\begin{center}
\begin{overpic}[width=.45\textwidth]{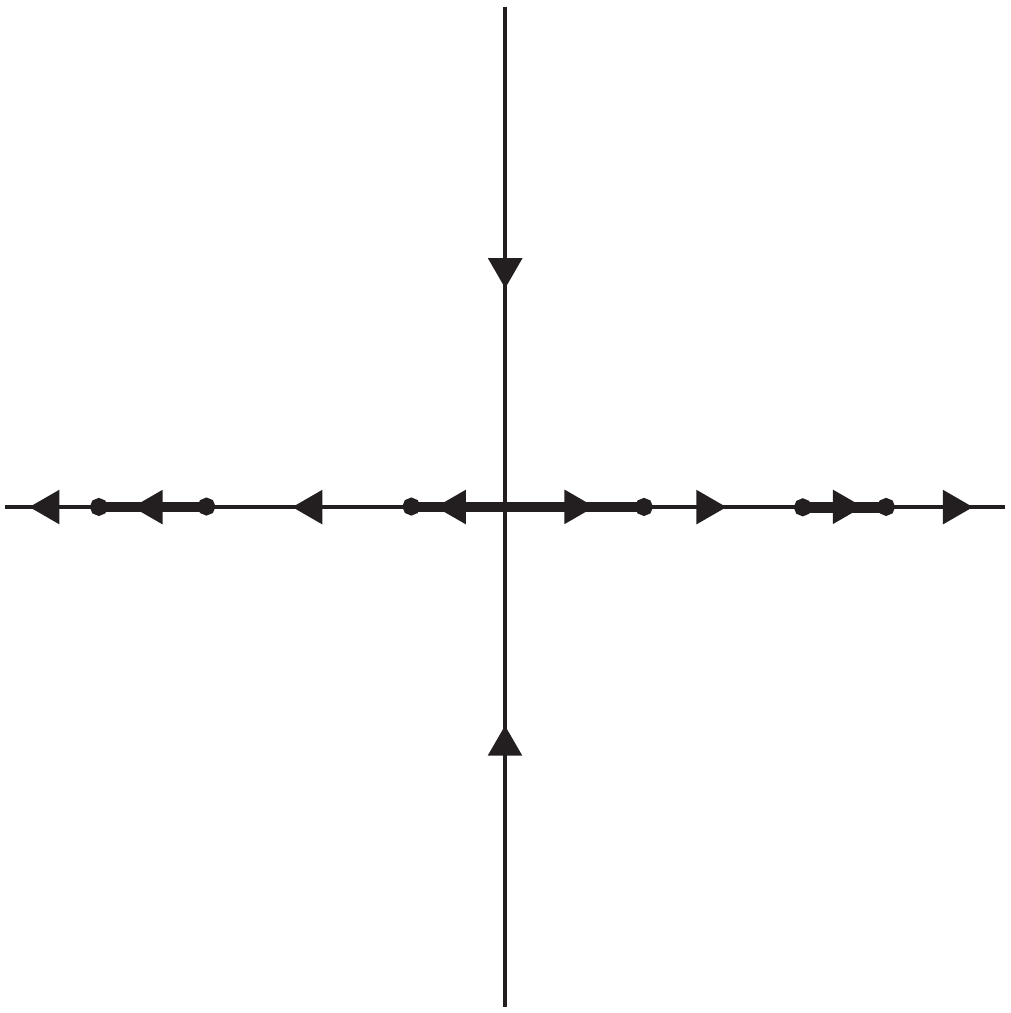}
      \put(74,76){\small $D_1$} 
      \put(20,76){\small $D_2$} 
      \put(20,22){\small $D_3$} 
      \put(74,22){\small $D_4$} 
      \put(52.5,72){\small $\tilde{v}_1$} 
      \put(92,53.5){\small $\tilde{v}_2$} 
      \put(54,43.5){\small $\tilde{v}_2^{\cut}$} 
      \put(68,53.5){\small $\tilde{v}_2$} 
      \put(80,43.5){\small $\tilde{v}_2^{\cut}$} 
      \put(52.5,25){\small $\tilde{v}_3$} 
      \put(3,53.5){\small $\tilde{v}_4$} 
      \put(12,43.5){\small $\tilde{v}_4^{\cut}$} 
      \put(29,53.5){\small $\tilde{v}_4$} 
      \put(41,43.5){\small $\tilde{v}_4^{\cut}$} 
     \end{overpic}\qquad
     \begin{overpic}[width=.45\textwidth]{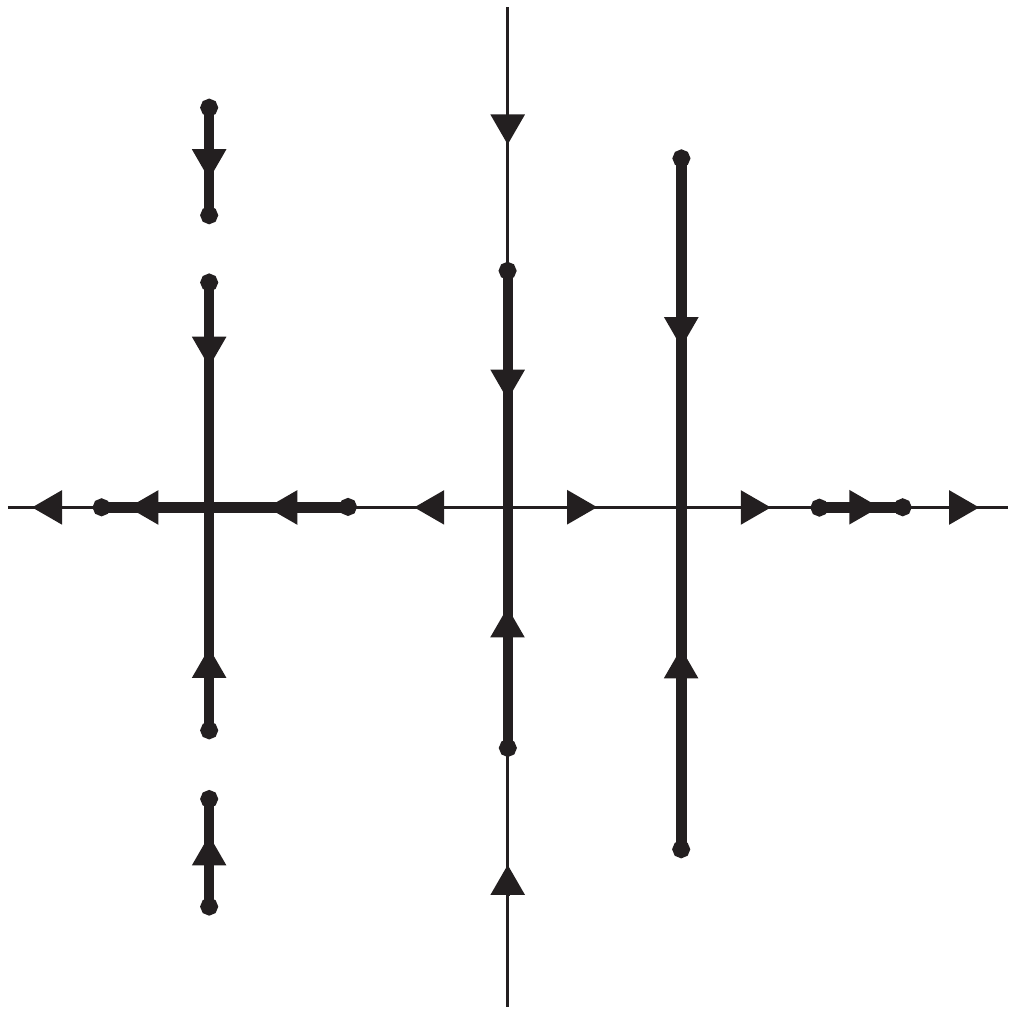}
      \put(84,76){\small $D_1$} 
      \put(7,76){\small $D_2$} 
      \put(7,22){\small $D_3$} 
      \put(84,22){\small $D_4$} 
      \put(52.5,86){\small $\tilde{v}_1$} 
      \put(52.5,62){\small $\tilde{v}_1^{\cut}$} 
      \put(69.5,67){\small $\tilde{v}_{D_1}^{\cut}$} 
      \put(55,53){\small $\tilde{v}_2$} 
      \put(72,53){\small $\tilde{v}_2$} 
      \put(82,53.5){\small $\tilde{v}_2^{\cut}$} 
      \put(23,83){\small $\tilde{v}_{D_2}^{\cut}$} 
      \put(23,65){\small $\tilde{v}_{D_2}^{\cut}$} 
      \put(23,33){\small $\tilde{v}_{D_3}^{\cut}$} 
      \put(23,14){\small $\tilde{v}_{D_3}^{\cut}$} 
      \put(93,53){\small $\tilde{v}_2$} 
      \put(52.5,12){\small $\tilde{v}_3$} 
      \put(52.5,37){\small $\tilde{v}_3^{\cut}$} 
      \put(2,53){\small $\tilde{v}_4$} 
      \put(11,53.5){\small $\tilde{v}_4^{\cut}$} 
      \put(26,53.5){\small $\tilde{v}_4^{\cut}$} 
      \put(69.5,33){\small $\tilde{v}_{D_4}^{\cut}$} 
      \put(41,53){\small $\tilde{v}_4$} 
     \end{overpic}
    \caption{\label{contourfig} Example of the contour $\tilde{\Sigma} = \R \cup i\R \cup \mathcal{C}$ and the associated jump matrices for $\lambda = 1$ (left) and $\lambda = -1$ (right). Dots indicate branch points and thick lines indicate the set of branch cuts $\mathcal{C}$.}
     \end{center}
\end{figure}

The above RH problem has a unique solution $\tilde{m}(x,t,k)$ for each $(x,t) \in [0,L] \times [0,\infty)$. The solution $q$ of the NLS equation (\ref{NLS}) can be obtained from  $\tilde{m}$ via the equation
\begin{align}\label{recoverq}
  q(x,t) = 2i\lim_{k \to \infty} k \tilde{m}_{12}(x,t,k), \qquad (x,t) \in [0,L] \times [0,\infty),
\end{align}
where the limit is taken along any ray $\{k \in \C | \arg k = \phi\}$ where $\phi \in \R \setminus \{n\pi/2 \, | \, n \in \Z\}$ (i.e., the ray is not contained in $\R \cup i\R$).
\end{theorem}
\begin{proof}
The NLS equation (\ref{NLS}) has a Lax pair given by
\begin{align}\label{laxpair}
  \mu_x+ik\hat{\sigma}_3 \mu=Q\mu, \qquad \mu_t+2ik^2 \hat{\sigma}_3 \mu=\tilde{Q}\mu,
\end{align}
where $k \in \C$ is the spectral parameter, $\mu(x,t,k)$ is a $2 \times 2$-matrix valued function, the matrices $Q$ and $\tilde{Q}$ are defined in terms of the solution $q(x,t)$ of (\ref{NLS}) by
\begin{align}
Q= \begin{pmatrix} 0 & q \\ \lambda \bar{q} & 0 \end{pmatrix}, \qquad \tilde{Q}=2kQ-i Q_x\sigma_3-i\lambda |q|^2\sigma_3,
\end{align}
and $\hat{\sigma}_3\mu = [\sigma_3, \mu]$.
Suppose $q(x,t)$ is a smooth solution of (\ref{NLS}) defined for  $(x,t) \in \R \times [0,T]$ such that $q(x+L, t) = q(x,t)$, where $L > 0$ and $0 < T < \infty$ is an arbitrary fixed final time.
Let $\mu_j(x,t,k)$, $j = 1,2,3,4$, denote the four solutions of (\ref{laxpair}) which are normalized to be the identity matrix at the points $(0,T)$, $(0,0)$, $(L,0)$, and $(L,T)$, respectively, see \cite{FI2004}. The spectral functions $a,b,A,B$ are defined for $k \in \C$ by
\begin{align}\label{abdef}
s(k)=\begin{pmatrix} \overline{a(\bar{k})} & b(k) \\ \lambda \overline{b(\bar{k})} & a(k) \end{pmatrix}, \quad
S(k)= \begin{pmatrix} \overline{A(\bar{k})} & B(k) \\ \lambda \overline {B(\bar{k})} & A(k) \end{pmatrix}, 
\end{align}
where 
$$s(k)=\mu_3(0,0,k), \qquad S(k)=\mu_1(0,0,k).$$ 
Note that $\{\mu_j\}_1^4$, $s$, and $S$ are entire functions of $k$. 
Clearly, $S(k)$ depends on $T$, whereas $s(k)$ does not. The entries of the matrices $s(k)$ and $S(k)$ are related by the so-called global relation, see \cite[Eq. (1.4)]{FI2004}:
\begin{align}\label{periodicgr}
(aA+\lambda \bar{b} e^{2ikL}B)B-(b A+\bar{a} e^{2ikL}B)A=e^{4ik^2T}c^+(k), \qquad k \in \C,
\end{align}
where $c^+(k)$ is an entire function such that
\begin{align}\label{cplusasymptotics}
c^+(k) = O\bigg(\frac{1}{k}\bigg) + O\bigg(\frac{e^{2ikL}}{k}\bigg), \qquad k \to \infty, \ k \in \C.
\end{align}
The above functions satisfy the unit determinant relations 
\begin{align}\label{det1relations}
  a \bar{a} - \lambda b \bar{b} = 1, \qquad A \bar{A} - \lambda B \bar{B} = 1.
\end{align}
As $k \to \infty$, the functions $a$ and $b$ satisfy
\begin{align}\label{abasymptotics}
a(k) = 1 + O\bigg(\frac{1}{k}\bigg) + O\bigg(\frac{e^{2ikL}}{k}\bigg), \qquad 
b(k) = O\bigg(\frac{1}{k}\bigg) + O\bigg(\frac{e^{2ikL}}{k}\bigg).
\end{align}
Using the entries of $s$ and $S$, we construct the following quantities:
\begin{align*}
& \alpha(k)= a A + \lambda \bar{b} B e^{2ikL}, \quad \beta(k)= b A + \bar{a} B e^{2ikL}, 
 \quad d(k)=a \bar{A}-\lambda b \bar{B}, \quad \delta(k)=\alpha \bar{A}-\lambda \beta \bar{B}.
\end{align*}
The eigenfunctions $\mu_j$ are related by
\begin{align}\label{murelations}
  \mu_3 = \mu_2 e^{-i\theta\hat{\sigma}_3} s(k), \qquad \mu_1 = \mu_2 e^{-i\theta\hat{\sigma}_3} S(k), \qquad
  \mu_4 = \mu_2 e^{-i\theta\hat{\sigma}_3}  \begin{pmatrix} \bar{\alpha} & \beta \\ 
  \lambda \bar{\beta} & \alpha \end{pmatrix}.
\end{align}

\begin{figure}
\begin{center}
\vspace{-.2cm}
\begin{overpic}[width=.4\textwidth]{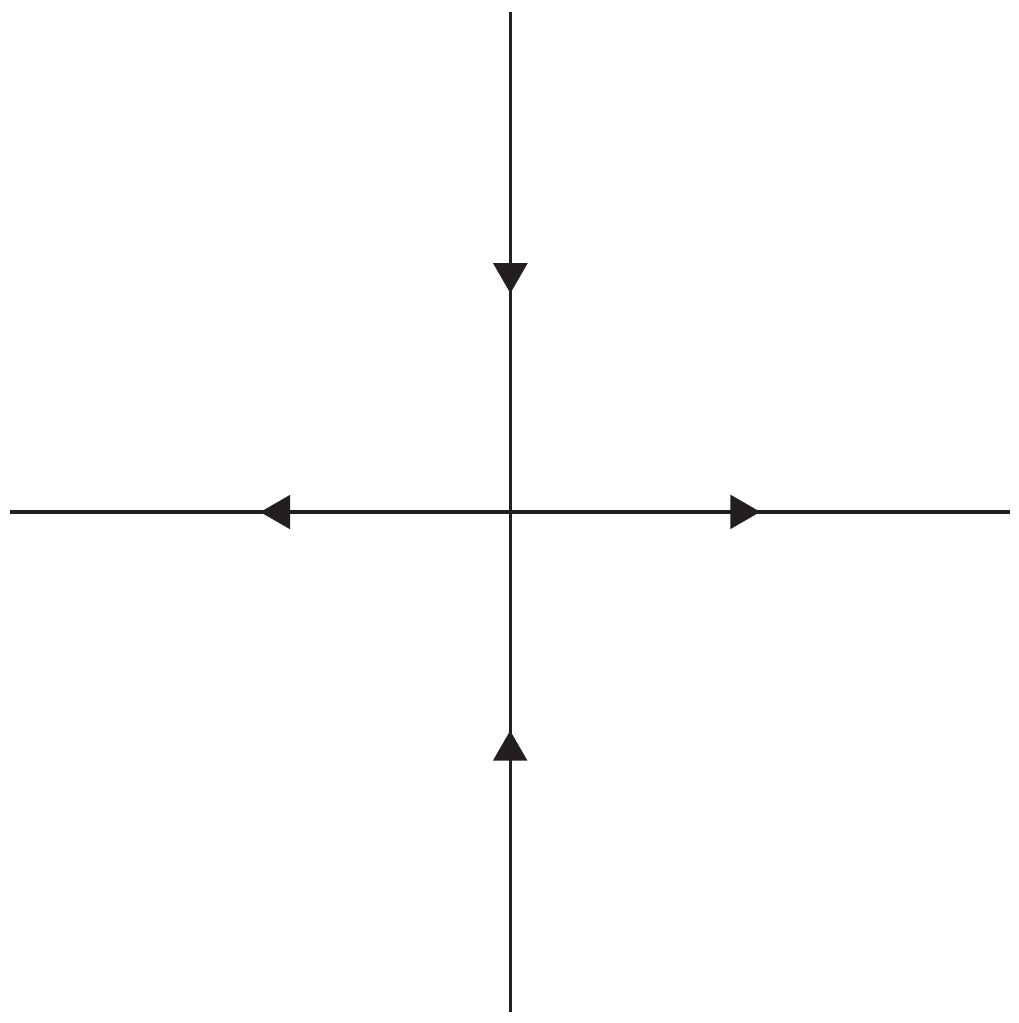}
      \put(74,76){\small $D_1$} 
      \put(20,76){\small $D_2$} 
      \put(20,22){\small $D_3$} 
      \put(74,22){\small $D_4$} 
       \put(70, 53.5){\small $J_2$} 
      \put(52, 71){\small $J_1$} 
      \put(25, 53.5){\small $J_4$} 
      \put(52, 25){\small $J_3$} 
     \end{overpic}
    \caption{\label{jumps.pdf} The four open quadrants $D_1, \dots D_4$, the contour $\R \cup i\R$, and the jump matrices $J_1, \dots, J_4$ defined in (\ref{J1234def}).}
     \end{center}
\end{figure}
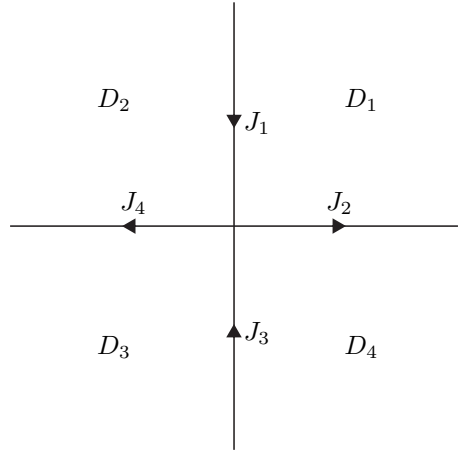

\subsection*{The RH problem for $M$}
It was shown in \cite{FI2004} that the sectionally meromorphic function  $M(x,t,k)$  defined by
\begin{align}\label{Mdef}
M = \begin{cases} 
(\frac{[\mu_2]_1}{\alpha}, [\mu_4]_2), \quad & k \in D_1, \\
(\frac{[\mu_1]_1}{d}, [\mu_3]_2), & k \in D_2, \\
([\mu_3]_1, \frac{[\mu_1]_2}{\bar{d}}), & k \in D_3, \\
([\mu_4]_1, \frac{[\mu_2]_2}{\bar{\alpha}}), & k \in D_4,
\end{cases}
\end{align}
satisfies
\begin{subequations}\label{rhp}
\begin{align}\label{rhpa}
& M_-(x,t,k)=M_+(x,t,k)J(x,t,k),\qquad k \in \R \cup i\R,
	\\
& M(x,t,k) = I + O(1/k), \qquad k \to \infty,
\end{align}
\end{subequations}
where the contour $\R \cup i\R$ is oriented as in Figure \ref{jumps.pdf} and
\begin{align}\label{J}
J=
\begin{cases}
J_2, & \arg k=0,\\
J_1, & \arg k=\pi/2,\\
J_4, & \arg k=\pi ,\\
J_3, & \arg k=3\pi/2,
\end{cases}
\end{align}
with
\begin{align}\nonumber
&J_1= \begin{pmatrix} \delta/d & -Be^{2ikL}e^{-2i\theta} \\ \lambda \bar{B} e^{2i\theta}/(d\alpha) & a/\alpha \end{pmatrix}, && 
 J_2= \begin{pmatrix} 1 & -\beta e^{-2i\theta}/\bar{\alpha} \\ \lambda \bar{\beta} e^{2i\theta}/\alpha & 1/(\alpha \bar{\alpha}) \end{pmatrix},
	\\  \label{J1234def}
& J_3= \begin{pmatrix} \bar{\delta}/\bar{d} & -Be^{-2i\theta}/\bar{d}\bar{\alpha} \\ \lambda \bar{B} e^{-2ikL} e^{2i\theta} & \bar{a}/\bar{\alpha} \end{pmatrix}, 
&&
 J_4=J_3 J_2^{-1} J_1.
\end{align}
The jump matrix $J$ depends on $x$ and $t$ only through the function $\theta(x,t,k)$ defined in (\ref{thetadef}).  
The solution $q(x,t)$ of (\ref{NLS}) can be recovered from  $M$ using
\begin{align}\label{sol}
q(x,t)=2i\lim_{k\rightarrow \infty} k M_{12}(x,t,k),
\end{align}
where the limit may be taken in any open quadrant.

If the functions $\alpha(k)$ and $d(k)$ have no zeros, then $M$ is analytic for $k \in \C \setminus (\R \cup i\R)$ and it can be characterized as the unique solution of the RH problem (\ref{rhp}). The jump matrix $J$  depends via the spectral functions on the initial data  $q(x,0)$ as well as on the boundary values $q(0,t)$ and $q_x(0,t)$. If all these boundary values are known, then the value of $q(x,t)$ at any point $(x,t)$  can be obtained by solving the RH problem (\ref{rhp}) for $M$ and using (\ref{sol}).
If the functions $\alpha(k)$ and $d(k)$ have zeros, then $M$ may have pole singularities and the RH problem has to be supplemented with appropriate residue conditions, see \cite[Proposition 2.3]{FI2004}.

\subsection*{The RH problem for $m$}
It is possible to formulate a RH problem which involves jump matrices defined in terms of the ratio $\Gamma = B/A$, as opposed to $A$ and $B$. Indeed, consider the sectionally meromorphic function $m(x,t,k)$ defined in terms of the eigenfunctions $\mu_j(x,t,k)$, $j = 1, \dots, 4$, by
\begin{align}\label{mdef}
m = \begin{cases} 
(\frac{A[\mu_2]_1}{\alpha}, \frac{[\mu_4]_2}{A}), \quad & k \in D_1, \\
(\frac{[\mu_1]_1}{d}, [\mu_3]_2), & k \in D_2, \\
([\mu_3]_1, \frac{[\mu_1]_2}{\bar{d}}), & k \in D_3, \\
(\frac{[\mu_4]_1}{\bar{A}}, \frac{\bar{A} [\mu_2]_2}{\bar{\alpha}}), & k \in D_4.
\end{cases}
\end{align}
The function $m$  is related to the solution $M$ defined in (\ref{Mdef}) by 
$$m = MH,$$
where the sectionally meromorphic function $H$ is defined by
$$H_1 = \begin{pmatrix} A & 0 \\ 0 & 1/A \end{pmatrix}, \qquad
H_2 = H_3 = I,\qquad
H_4 = \begin{pmatrix} 1/\bar{A} & 0 \\ 0 & \bar{A} \end{pmatrix},$$
with $H_j$ denoting the restriction of $H$ to $D_j$, $j = 1,2,3,4$. 

The limits of $m(x,t,k)$ as $k$ approaches $(\R \cup i\R) \setminus \{0\}$ from the left and right satisfy
\begin{align}\label{mjump}
  m_-(x,t,k) = m_+(x, t, k) v(x, t, k), \qquad k \in (\R \cup i\R) \setminus \{0\},
\end{align}
where the jump matrix $v$ is defined by
\begin{align}\nonumber
&  v_1 = \begin{pmatrix}  \frac{a  - \lambda b \bar{\Gamma} - \lambda \Gamma (\bar{a}\bar{\Gamma} - \bar{b})e^{2ikL} }{a- \lambda b \bar{\Gamma}} & -\Gamma  e^{2ikL} e^{-2 i \theta } \\
 \frac{\lambda \bar{\Gamma} e^{2 i \theta } }{(a - \lambda b \bar{\Gamma}) (a + \lambda\bar{b} \Gamma  e^{2ikL})} & \frac{a}{a + \lambda \bar{b} \Gamma  e^{2ikL}} \end{pmatrix}, && \arg k = \frac{\pi}{2},
	\\ \nonumber
& v_2 =  \begin{pmatrix} 1 - \lambda \Gamma \bar{\Gamma} 
& -\frac{(\bar{a} \Gamma e^{2ikL}+b) e^{-2 i \theta }}{\bar{a}+ \lambda b \bar{\Gamma} e^{-2ikL}} \\
 \frac{\lambda (a \bar{\Gamma}e^{-2ikL}+\bar{b}) e^{2 i \theta } }{a+ \lambda \bar{b} \Gamma  e^{2ikL} } &
   \frac{1}{(a+ \lambda \bar{b} \Gamma  e^{2ikL}) (\bar{a} + \lambda b \bar{\Gamma} e^{-2ikL})}  \end{pmatrix},&& \arg k = 0,
   	\\  \nonumber
& v_3 = \begin{pmatrix} \frac{\bar{a}  - \lambda \bar{b} \Gamma - \lambda \bar{\Gamma}(a\Gamma - b) e^{-2ikL}}{\bar{a} - \lambda \bar{b} \Gamma } & - \frac{\Gamma e^{-2 i \theta}}{(\bar{a} - \lambda \bar{b} \Gamma) (\bar{a} + \lambda b \bar{\Gamma} e^{-2ikL} )} \\
 \lambda \bar{\Gamma} e^{-2ikL} e^{2 i \theta} & \frac{\bar{a}}{\bar{a} + \lambda b \bar{\Gamma} e^{-2ikL}}  \end{pmatrix}, && \arg k = -\frac{\pi}{2},
   	\\ \label{vdef}
& v_4 = \begin{pmatrix}	
 \frac{1 - \lambda \Gamma \bar{\Gamma}}{(a - \lambda b \bar{\Gamma}) (\bar{a} - \lambda \bar{b} \Gamma )} & -\frac{(a \Gamma - b)e^{-2 i \theta } }{\bar{a} - \lambda \bar{b} \Gamma } \\
 \frac{\lambda(\bar{a} \bar{\Gamma} - \bar{b})e^{2 i \theta }}{a - \lambda b \bar{\Gamma}} & 1
 \end{pmatrix}, && \arg k = \pi.
\end{align}
Furthermore, 
\begin{subequations}\label{RHmconditions}
\begin{align}
& \text{$m(x,t,k) = I + O(k^{-1})$ as $k \to \infty$,}
	\\
& \text{$m(x,t,k) = O(1)$ as $k \to 0$.}
\end{align}
\end{subequations}
Also, the function $m$ satisfies $\det m = 1$, as well as the symmetry relations
\begin{align}\label{msymm}
  m_{11}(x,t,k) = \overline{m_{22}(x,t,\bar{k})}, \qquad 
  m_{21}(x,t,k) = \lambda \overline{m_{12}(x,t,\bar{k})}.
\end{align}

\subsection*{The RH problem for $\tilde{m}$}
In order to define a RH problem which depends on $\tilde{\Gamma}$ instead of $\Gamma$, where $\tilde{\Gamma}$ is independent of $T$, we introduce the function $g(k)$ by
\begin{align}\nonumber
& g_1 = \begin{pmatrix} \frac{a + \lambda \bar{b} \Gamma e^{2ikL}}{a + \lambda \bar{b} \tilde{\Gamma} e^{2ikL}} & (\tilde{\Gamma} - \Gamma) e^{2ikL} e^{-2i\theta} \\ 0 & \frac{a + \lambda \bar{b} \tilde{\Gamma} e^{2ikL}}{a + \lambda \bar{b}\Gamma e^{2ikL}} \end{pmatrix}, \qquad
g_2 = \begin{pmatrix} 1 & 0 \\ 
\frac{\lambda (\bar{\tilde{\Gamma}} - \bar{\Gamma})e^{2i\theta}}{(a - \lambda b \bar{\Gamma})(a - \lambda b \bar{\tilde{\Gamma}})} & 1 \end{pmatrix},
	\\ \label{gdef}
& g_3 = \begin{pmatrix} 1 & \frac{(\tilde{\Gamma} - \Gamma)e^{-2i\theta}}{(\bar{a} - \lambda \bar{b} \Gamma)(\bar{a} - \lambda \bar{b} \tilde{\Gamma})} \\
0 & 1 \end{pmatrix}, \qquad
g_4 = \begin{pmatrix} \frac{\bar{a} + \lambda b \bar{\tilde{\Gamma}} e^{-2ikL}}{\bar{a} + \lambda b \bar{\Gamma} e^{-2ikL}} & 0 \\
\lambda(\bar{\tilde{\Gamma}} - \bar{\Gamma}) e^{-2ikL} e^{2i\theta} & \frac{\bar{a} + \lambda b \bar{\Gamma} e^{-2ikL}}{\bar{a} + \lambda b \bar{\tilde{\Gamma}} e^{-2ikL}} \end{pmatrix},
\end{align} 
where $g_j$ denotes the restriction of $g$ to  $D_j$, $j = 1, \dots, 4$.
In (\ref{gdef}), $\Gamma = B/A$ and $\tilde{\Gamma}$ is defined by (\ref{tildeGammadef}).
We define $\tilde{m}(x,t,k)$ by
\begin{align}\label{tildemdef}
\tilde{m} = mg.
\end{align}
Let $\mathcal{C}$ be the set of branch cuts of $\sqrt{4 - \Delta^2}$ introduced in Theorem \ref{mainth}. The function $\sqrt{4 - \Delta^2}$ is analytic and single-valued on $\C \setminus \mathcal{C}$; the branch is fixed by the condition that $\sqrt{4 - \Delta^2} = 2\sin(kL)(1 + O(k^{-1}))$ as $k \in \C  \setminus \cup_{n\in \Z} \mathcal{D}_n$ tends to infinity, see \cite{FLxperiodic} for details.

A straightforward computation shows that $\tilde{m}$ satisfies the jump relation (\ref{mtildejump}) on $(\R \cup i\R) \setminus \mathcal{C}$, where $\tilde{v}$ is given by (\ref{vdef}) with $\Gamma$  replaced by $\tilde{\Gamma}$.
However, since the square root $\sqrt{4 - \Delta^2}$ changes sign as $k$ crosses a branch cut, $\tilde{v}$ is not given by the same expression as $v$ on $(\R \cup i\R) \cap \mathcal{C}$. Also, if $\lambda = -1$, then $\tilde{m}$ may have jumps across the contours $\mathcal{C} \cap D_j$, $j = 1, \dots, 4$. Remarkably, all these jumps can be expressed in terms of $a$ and $b$, yielding (\ref{tildevdef}).

The proof of uniqueness of the solution $\tilde{m}$ of the RH problem of Theorem \ref{mainth} uses the fact that the problem with $\tilde{P} \neq \emptyset$ can be transformed into a problem with $\tilde{P} = \emptyset$ and the fact that $\det \tilde{m}$ is an entire function, see \cite{FLxperiodic}.

Fix $(x,t) \in [0,L] \times [0,\infty)$. Choose $T \in (t, \infty)$ and define $\tilde{m}$ by (\ref{tildemdef}) with $m$ and $g$ defined using $T$ as final time. Straightforward computations show that  $\tilde{m} = MHg$ satisfies the RH problem of Theorem \ref{mainth} and that (\ref{recoverq}) holds at $(x,t)$ (the verification of (\ref{recoverq}) uses (\ref{sol}) and the fact that $g - I$ is exponentially small as $k \to \infty$ along any ray $\arg k = \text{constant}$ not contained in $\R \cup i\R$; details are given in \cite{FLxperiodic}).
\end{proof}

\begin{remark}
The solutions obtained via algebraic geometry correspond to the particular case of the above formalism when the number of branch cuts of $\sqrt{4 - \Delta^2}$ is finite. This implies that in these particular cases the associated RH problems can be solved via Riemann theta functions. 
\end{remark}

\section{Example: Constant initial data}\label{constantexamplesec}
To illustrate the approach of Theorem \ref{mainth}, we consider the case of constant initial data
$$q(x,0) = q_0, \qquad x \in [0,L],$$
where $q_0 > 0$ can be taken to be positive thanks to the phase invariance of (\ref{NLS}). 
In this case, direct integration of the $x$-part of the Lax pair (\ref{laxpair}) leads to the following expressions for the spectral functions $a$ and $b$:
\begin{align}\label{abexample}
a(k) = e^{ikL}\bigg(\cos(Lr(k)) - \frac{ik\sin(Lr(k))}{r(k)}\bigg), \qquad
b(k) = -\frac{q_0 e^{ikL}\sin(Lr(k))}{r(k)},
\end{align}
where $r(k)$ denotes the square root
$$r(k) = \sqrt{k^2 - \lambda q_0^2}.$$
It follows that
$$\Delta(k) = 2\cos(L r(k)).$$
Note that $a$, $b$, and $\Delta$ are entire functions of $k$ even though $r(k)$  has a branch cut.

The function $4 - \Delta^2 = 4\sin^2(Lr)$ has the two simple zeros
$$\lambda^\pm = \begin{cases} \pm q_0 & \text{if $\lambda = 1$}, \\
\pm iq_0 & \text{if $\lambda = -1$},
\end{cases}$$
as well as the infinite sequence of double zeros
$$\pm \frac{\sqrt{n^2 \pi^2 + \lambda L^2 q_0^2}}{L}, \qquad n \in \Z \setminus \{0\}.$$
If $\lambda = 1$, then all zeros are real; if $\lambda = -1$, then the zeros are real for $|n|  \geq L q_0/\pi$ and purely imaginary for $|n| < L q_0/\pi$.
The function $\sqrt{4 - \Delta^2}$ is single-valued on $\C \setminus \mathcal{C}$, where $\mathcal{C}$ consists of a single branch cut:
$$\mathcal{C} = (\lambda^-, \lambda^+).$$
Let us fix the branch in the definition of $r$ so that $r:\C \setminus \mathcal{C} \to \C$ is analytic and $r(k) \sim k$ as $k \to \infty$.
Then 
$$\sqrt{4 - \Delta^2} = 2\sin(Lr),$$
and hence the function $\tilde{\Gamma}:\C \setminus \mathcal{C} \to \C$ defined in (\ref{tildeGammadef}) is given by
\begin{align}\label{tildeGammaexample}
 \tilde{\Gamma}(k) = -\frac{\lambda i (k - r(k))}{q_0}.
\end{align}
The function $\tilde{\Gamma}(k)$ has no poles. The contour  $\tilde{\Sigma}$ is given by $\tilde{\Sigma} = \R \cup i\R$ and is oriented as in Figure \ref{jumps.pdf}; the origin is its only point of self-intersection and so $\tilde{\Sigma}_\star = (\R \cup i\R) \setminus \{0, \lambda^\pm\}$.
Substituting the expressions (\ref{abexample}) and (\ref{tildeGammaexample}) for $a,b,\tilde{\Gamma}$ into the definition (\ref{tildevdef}) of the jump matrix $\tilde{v}$, we get
\begin{subequations}\label{tildevjexample}
\begin{align}
& \tilde{v}_1 = \begin{pmatrix} 
 \frac{2 \lambda  r (k-r)}{q_0^2} & \frac{i \lambda 
   (k-r) e^{-2 i (\theta - k L)}}{q_0} \\
 \frac{i (k-r) e^{2 i (\theta - k L + L r)}}{q_0} & \frac{e^{i L
   r} (r \cos(L r)-i k \sin(L r))}{r}
 \end{pmatrix}, &&  k \in i\R_+ \setminus \mathcal{C},
	\\ 
& \tilde{v}_2 =  \begin{pmatrix} 
 \frac{2 \lambda  r (k-r)}{q_0^2} & \frac{i \lambda 
   (k-r) e^{-2 i (\theta - k L + L r)}}{q_0} \\
 \frac{i (k-r) e^{2 i (\theta - k L + L r)}}{q_0} & 1 
  \end{pmatrix}, && k \in \R_+ \setminus \mathcal{C},
   	\\  
& \tilde{v}_3 = \begin{pmatrix} 
 \frac{2 \lambda  r (k-r)}{q_0^2} & \frac{i \lambda 
   (k-r) e^{-2 i (\theta -k L + L r)}}{q_0} \\
 \frac{i (k-r) e^{2i(\theta-kL)}}{q_0} & \frac{e^{-i L r}
   (r \cos(L r)+i k \sin(L r))}{r} 
     \end{pmatrix}, && k \in i\R_- \setminus \mathcal{C},
   	\\ 
&  \tilde{v}_4 = \begin{pmatrix}	
  \frac{2 \lambda  r (k-r)}{q_0^2} & \frac{i \lambda 
   (k-r) e^{-2 i (\theta - k L)}}{q_0} \\
 \frac{i (k-r) e^{2i(\theta-kL)}}{q_0} & 1 
 \end{pmatrix}, && k \in \R_- \setminus \mathcal{C}.
\end{align}
\end{subequations}
If $\lambda = 1$, then $\mathcal{C} = (-q_0, q_0)$ so the formulation of the RH problem also involves the jump matrices $\tilde{v}_2^{\cut}$ and $\tilde{v}_4^{\cut}$, whereas if $\lambda = -1$, then $\mathcal{C} = (-iq_0, iq_0)$ so the formulation instead involves the jump matrices $\tilde{v}_1^{\cut}$ and $\tilde{v}_3^{\cut}$. 
Let $\mathfrak{r}(k) = \sqrt{|k^2 - \lambda q_0^2|} \geq 0$.
A calculation shows that if $\lambda = 1$, then
\begin{subequations}\label{tildev24cutexample}
\begin{align}
& \tilde{v}_2^{\cut} = 
\begin{pmatrix}
0 & \frac{(\mathfrak{r}(k)+i k) e^{-2 i (\theta - k L)}}{q_0} \\
 -\frac{(\mathfrak{r}(k)-i k) e^{2i(\theta-kL)}}{q_0} &
   e^{-2 L \mathfrak{r}(k)} \\
\end{pmatrix},
&& k \in \mathcal{C} \cap \R_+,	
	\\
& \tilde{v}_4^{\cut} = \begin{pmatrix}
 0 & -\frac{(\mathfrak{r}(k)-i k) e^{-2 i (\theta - k L)}}{q_0} \\
 \frac{(\mathfrak{r}(k)+i k) e^{2i(\theta-kL)}}{q_0} & 1 
\end{pmatrix},
&& k \in \mathcal{C} \cap \R_-,
\end{align}
\end{subequations}
where $\mathfrak{r}(k) = \sqrt{q_0^2 - k^2} \geq 0$. If $\lambda=-1$, then
\begin{subequations}\label{tildev13cutexample}
\begin{align}
& \tilde{v}_1^{\cut} = \begin{pmatrix}
 0 & \frac{i (\mathfrak{r}(k)-k) e^{-2 i (\theta - k L)}}{q_0} \\
 \frac{i (\mathfrak{r}(k)+k) e^{2i(\theta-kL)}}{q_0} &
   \frac{1+e^{2 i L \mathfrak{r}(k)}}{2}+\frac{k (1-e^{2 i L
   \mathfrak{r}(k)})}{2 \mathfrak{r}(k)} 
\end{pmatrix}, && k \in \mathcal{C} \cap i\R_+,
	\\
& \tilde{v}_3^{\cut} = \begin{pmatrix}
 0 & -\frac{i (\mathfrak{r}(k)+k) e^{-2 i (\theta - k L )}}{q_0} \\
 -\frac{i (\mathfrak{r}(k)-k) e^{2i(\theta-kL)}}{q_0} &
   \frac{1+e^{-2 i L \mathfrak{r}(k)}}{2} +\frac{k (1-e^{-2 i L
   \mathfrak{r}(k)})}{2 \mathfrak{r}(k)}
\end{pmatrix}, && k \in \mathcal{C} \cap i\R_-,
\end{align}
\end{subequations}
where $\mathfrak{r}(k) = \sqrt{q_0^2 + k^2} \geq 0$. We conclude that in the case of constant initial data, the RH problem for $\tilde{m}$ can be formulated as follows (recall that the contour $\R \cup i\R$ is oriented as in Figure \ref{jumps.pdf}; see also Figure \ref{constantexamplefig}).

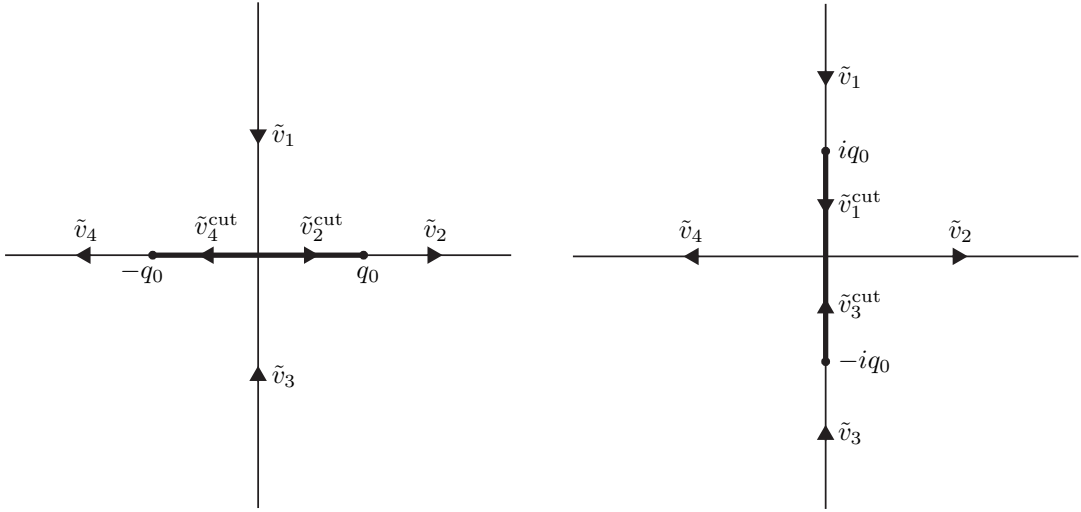
\begin{figure}
\begin{center}
\begin{overpic}[width=.45\textwidth]{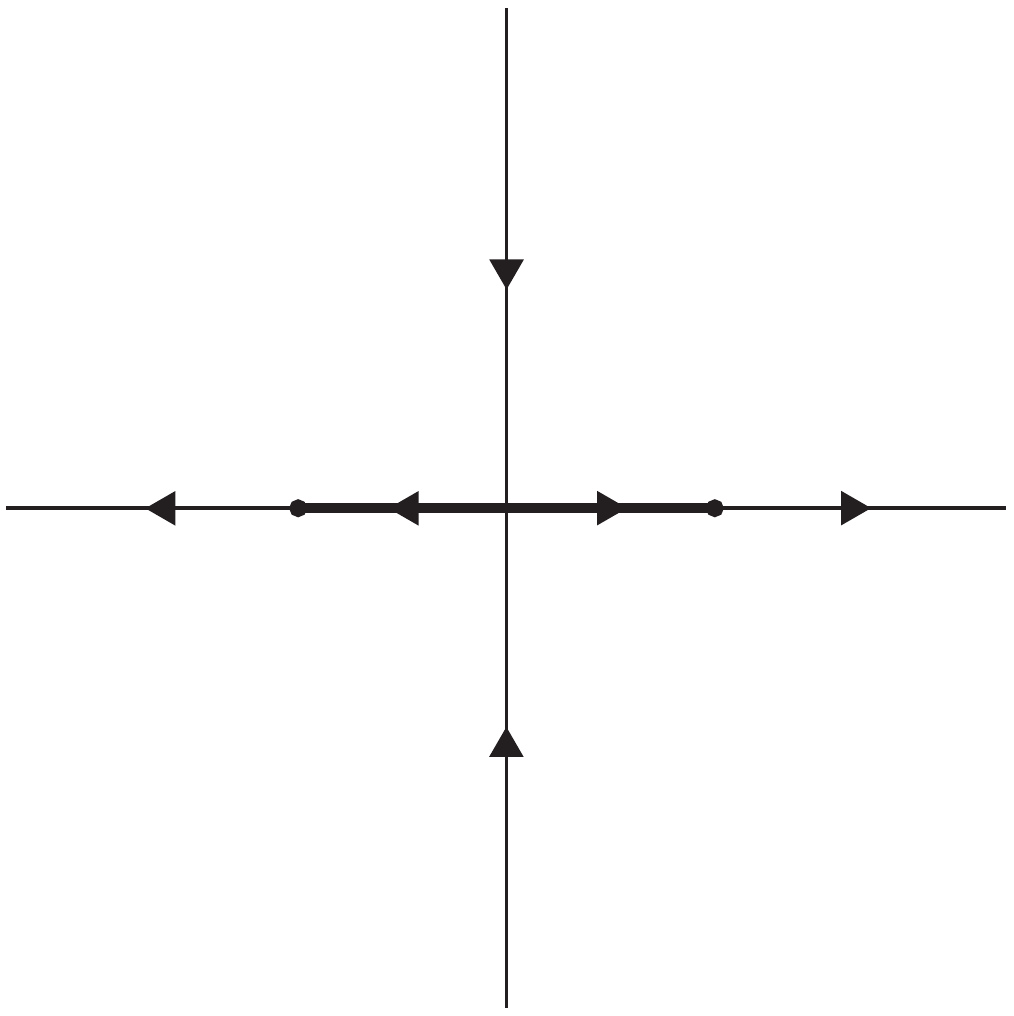}
      \put(23,45.5){\small $-q_0$} 
      \put(69,45.5){\small $q_0$} 
      \put(52.5,72){\small $\tilde{v}_1$} 
      \put(82,53.5){\small $\tilde{v}_2$} 
      \put(58,54){\small $\tilde{v}_2^{\cut}$} 
      \put(52.5,25){\small $\tilde{v}_3$} 
      \put(14,53.5){\small $\tilde{v}_4$} 
      \put(37.5,54){\small $\tilde{v}_4^{\cut}$} 
     \end{overpic}\qquad
     \begin{overpic}[width=.45\textwidth]{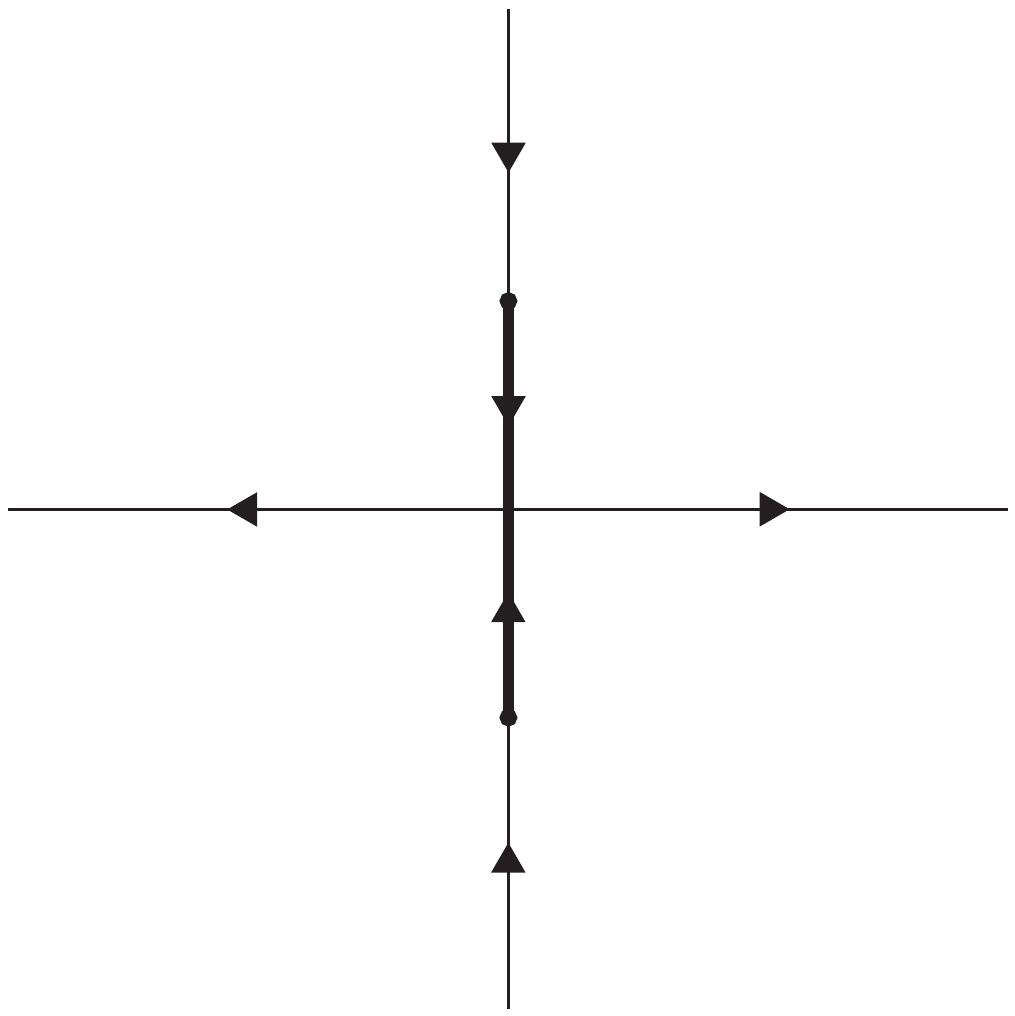}
      \put(52.5,28){\small $-iq_0$} 
      \put(52.5,69){\small $iq_0$} 
      \put(52.5,83.5){\small $\tilde{v}_1$} 
      \put(52.5,58.5){\small $\tilde{v}_1^{\cut}$} 
      \put(74,53.5){\small $\tilde{v}_2$} 
      \put(52.5,14){\small $\tilde{v}_3$} 
      \put(52.5,39){\small $\tilde{v}_3^{\cut}$} 
      \put(21.5,53.5){\small $\tilde{v}_4$} 
     \end{overpic}
    \caption{\label{constantexamplefig} The jump matrices for the RH problem \ref{RHmtildeexample} for $\lambda = 1$ (left) and $\lambda = -1$ (right).}
     \end{center}
\end{figure}

\begin{RHproblem}{\bf(The RH problem for constant initial data)}\label{RHmtildeexample}
Find a $2 \times 2$-matrix valued function $\tilde{m}(x,t,k)$ with the following properties:
\begin{itemize}
\item $\tilde{m}(x,t,\cdot) : \C \setminus (\R \cup i\R) \to \C^{2 \times 2}$ is analytic.

\item The limits of $m(x,t,k)$ as $k$ approaches $(\R \cup i\R) \setminus \{0, \lambda^\pm\}$ from the left and right exist, are continuous on $(\R \cup i\R) \setminus \{0, \lambda^\pm\}$, and satisfy
\begin{align}\label{mtildejumpexample}
  \tilde{m}_-(x,t,k) = \tilde{m}_+(x, t, k) \tilde{v}(x, t, k), \qquad k \in (\R \cup i\R) \setminus \{0, \lambda^\pm\},
\end{align}
where $\tilde{v}$ is defined as follows (see Figure \ref{constantexamplefig}):
\begin{itemize}
\item If $\lambda = 1$, then
$$\tilde{v} = \begin{cases} 
\tilde{v}_1 & \text{on $i\R_+$},			\\
\tilde{v}_2 & \text{on $(q_0, +\infty)$},	\\
\tilde{v}_2^{\cut} & \text{on $(0, q_0)$},	\\
\tilde{v}_3 & \text{on $i\R_-$},			\\
\tilde{v}_4 & \text{on $(-\infty, -q_0)$},	\\
\tilde{v}_4^{\cut} & \text{on $(-q_0, 0)$},
\end{cases}$$
where $\{\tilde{v}_j\}_1^4$, $\tilde{v}_2^{\cut}$, and $\tilde{v}_4^{\cut}$
are defined by (\ref{tildevjexample}) and (\ref{tildev24cutexample}).

\item If $\lambda = -1$, then
$$\tilde{v} = \begin{cases} 
\tilde{v}_1 & \text{on $(iq_0, +\infty)$},	\\
\tilde{v}_1^{\cut} & \text{on $(0, iq_0)$},	\\
\tilde{v}_2 & \text{on $\R_+$},			\\
\tilde{v}_3 & \text{on $(-i\infty, -iq_0)$},	\\
\tilde{v}_3^{\cut} & \text{on $(-iq_0, 0)$}, 	\\
\tilde{v}_4 & \text{on $\R_-$},
\end{cases}$$
where $\{\tilde{v}_j\}_1^4$, $\tilde{v}_1^{\cut}$, and $\tilde{v}_3^{\cut}$
are defined by (\ref{tildevjexample}) and (\ref{tildev13cutexample}).

\end{itemize}

\item $\tilde{m}(x,t,k) = I  + O\big(k^{-1}\big)$ as $k \to \infty$, $k \in \C \setminus \cup_{n\in \Z} \mathcal{D}_n$.

\item $\tilde{m}(x,t,k) = O(1)$ as $k \to \{0, \lambda^\pm\}$, $k \in \C \setminus (\R \cup i\R)$.

\end{itemize}
\end{RHproblem}

\begin{remark}
It is easy to verify that the jump matrices in RH problem \ref{RHmtildeexample}  satisfy the following consistency conditions at the origin:
\begin{align*}
\begin{cases}
 (\tilde{v}_4^{\cut} )^{-1} \tilde{v}_3 (\tilde{v}_2^{\cut} )^{-1} \tilde{v}_1\big|_{k=0} = I, & \lambda = 1,
	\\
\tilde{v}_4^{-1} \tilde{v}_3^{\cut} \tilde{v}_2^{-1} \tilde{v}_1^{\cut} \big|_{k=0} = I, & \lambda = -1.
\end{cases}
\end{align*}
\end{remark}

\subsection{Solution of the RH problem for $\tilde{m}$}\label{constantexamplesubsec}
In what follows, we solve the RH problem \ref{RHmtildeexample} for $\tilde{m}$ explicitly by transforming it to a RH problem which has a constant off-diagonal jump across the branch cut $\mathcal{C} = (\lambda^-, \lambda^+)$.

The jump matrices $\tilde{v}_1$  and $\tilde{v}_3$ in (\ref{tildevjexample}) admit the factorizations
\begin{align*}
 & \tilde{v}_1 = \begin{pmatrix}
   1 & 0 \\
 \frac{i q_0 e^{2 i (\theta - k L + L r)}}{2 \lambda  r} & 1 
  \end{pmatrix}\begin{pmatrix}
   \frac{2 \lambda  r (k-r)}{q_0^2} & \frac{i \lambda 
   (k-r) e^{-2 i (\theta - k L)}}{q_0} \\
 0 & \frac{k+r}{2 r}
  \end{pmatrix},
  	\\
&  \tilde{v}_3 =  \begin{pmatrix}  
   1 & 0 \\
 \frac{i q_0 e^{2 i(\theta - k L)}}{2 \lambda  r} & 1 
   \end{pmatrix}\begin{pmatrix}
 \frac{2 \lambda  r (k-r)}{q_0^2} & \frac{i \lambda 
   (k-r) e^{-2 i (\theta -k L+L r)}}{q_0} \\
 0 & \frac{k+r}{2 r} 
   \end{pmatrix}.
\end{align*} 
It follows that the jump across $i\R \setminus \mathcal{C}$ can be removed by introducing a new solution $\hat{m}$ by
$$\hat{m} = \tilde{m} \times \begin{cases}
\begin{pmatrix}
 1 & 0 \\
 \frac{i q_0 e^{2 i (\theta -k L+L r)}}{2 \lambda  r} & 1
 \end{pmatrix}, & k \in D_1, 
 	\\
 \begin{pmatrix}
 \frac{2 \lambda  r (k-r)}{q_0^2} & \frac{i \lambda 
   (k-r) e^{-2 i (\theta - k L)}}{q_0} \\
 0 & \frac{k+r}{2 r}
 \end{pmatrix}^{-1}, & k \in D_2, 
 	\\
\begin{pmatrix}
 1 & 0 \\
 \frac{i q_0 e^{2 i (\theta - k L)}}{2 \lambda  r} & 1
 \end{pmatrix}, & k \in D_3, 
 	\\
\begin{pmatrix}
 \frac{2 \lambda  r (k-r)}{q_0^2} & \frac{i \lambda 
   (k-r) e^{-2 i (\theta -k L+L r)}}{q_0} \\
 0 & \frac{k+r}{2 r} 
 \end{pmatrix}^{-1}, & k \in D_4.
\end{cases}
$$
Straightforward computations using (\ref{tildevjexample})-(\ref{tildev13cutexample}) show that $\hat{m}$  only has a jump across the cut $(\lambda^-, \lambda^+)$.
Let us orient
$$(\lambda^-, \lambda^+) = \begin{cases} (-q_0,q_0), & \lambda = 1, \\ 
(-iq_0, iq_0), & \lambda = -1, \end{cases}$$
to the right if $\lambda = 1$  and upward if $\lambda = -1$.
We find that $\hat{m}(x,t,\cdot) : \C \setminus (\lambda^-, \lambda^+) \to \C^{2 \times 2}$ is analytic, that $\hat{m} = I  + O(k^{-1})$ as $k \to \infty$, and that $\hat{m}$  satisfies the jump condition
\begin{align}\label{mhatjump}
\hat{m}_- = \hat{m}_+ \begin{pmatrix}  
 0 & f  \\
-1/f & 0 
  \end{pmatrix} \quad \text{for $k \in (\lambda^-, \lambda^+)$},
\end{align}  
where $f(k) \equiv f(x,t,k)$ is defined by
$$f(k) = -\frac{2 i \lambda r_+(k)}{q_0} e^{-2 i (\theta - k L)}
= \begin{cases} 
\frac{2 \sqrt{|q_0^2-k^2|} e^{-2 i (\theta - k L)}}{q_0}, & \lambda = 1,	\\
-\frac{2 i \sqrt{|k^2+q_0^2|} e^{-2 i (\theta - k L)}}{q_0}, & \lambda = -1.
\end{cases}$$

The jump matrix in (\ref{mhatjump}) can be made constant (i.e., independent of $k$) by performing another transformation. Define $\delta(k) \equiv \delta(x,t,k)$ by
\begin{align}\label{deltadef}
\delta(k) = \exp\bigg(\frac{r(k)}{2\pi i} \int_{\lambda^-}^{\lambda^+} \frac{\ln f(s)}{r_+(s) (s-k)}ds\bigg), \qquad k \in \C \setminus (\lambda^-, \lambda^+).
\end{align}
The function $\delta$ satisfies the jump relation $\delta_+ \delta_- = f$ on $(\lambda^-, \lambda^+)$ and $\lim_{k\to \infty} \delta(k) = \delta_\infty$, where
\begin{align}\label{deltainftydef}
\delta_\infty = \exp\bigg(-\frac{1}{2\pi i} \int_{\lambda^-}^{\lambda^+} \frac{\ln f(s)}{r_+(s)}ds\bigg).
\end{align}
Moreover, $\delta(k) = O((k-\lambda^\pm)^{1/4})$ and $\delta(k)^{-1} = O((k-\lambda^\pm)^{-1/4})$ as $k \to \lambda^\pm$.
Consequently, $\check{m} = \delta_\infty^{-\sigma_3} \hat{m} \delta^{\sigma_3}$ satisfies the following RH problem: (i) $\check{m}(x,t,\cdot) : \C \setminus (\lambda^-, \lambda^+) \to \C^{2 \times 2}$ is analytic, (ii) $\check{m} = I  + O(k^{-1})$ as $k \to \infty$, (iii) $\check{m} = O((k-\lambda^\pm)^{-1/4})$ as $k \to \lambda^\pm$, and (iv) $\check{m}$  satisfies the jump condition
$$\check{m}_- = \check{m}_+ \begin{pmatrix}  
 0 & 1  \\
-1 & 0 
  \end{pmatrix} \quad \text{for $k \in (\lambda^-, \lambda^+)$}.$$
The unique solution of this RH problem is given explicitly by
\begin{align}\label{mcheckexplicit}
\check{m} = \frac{1}{2}\begin{pmatrix} Q + Q^{-1} & i(Q - Q^{-1}) \\ -i(Q - Q^{-1}) & Q + Q^{-1} \end{pmatrix} \quad \text{with} \ \  Q(k) = \bigg(\frac{k-\lambda^+}{k - \lambda^-}\bigg)^{1/4},
\end{align}
where the branch of $Q: \C \setminus (\lambda^-, \lambda^+) \to \C$ is such that $Q \sim 1$ as $k \to \infty$. Since $\tilde{m}$ is easily obtained from $\check{m}$  by inverting the transformations $\tilde{m} \to \hat{m} \to \check{m}$, this provides an explicit solution of the RH problem for $\tilde{m}$. 

We next use the formula (\ref{mcheckexplicit}) for $\check{m}$ together with (\ref{recoverq}) to find $q(x,t)$. By (\ref{recoverq}),
\begin{align}\label{qfrommcheck}
q(x,t) 
= 2i\lim_{k \to \infty} k \hat{m}_{12}(x,t,k) 
= 2i \delta_\infty^{2} \lim_{k \to \infty} k \check{m}_{12}(x,t,k) 
= \frac{\lambda^+ - \lambda^-}{2} \delta_\infty^{2}.
\end{align}
In order to compute $\delta_\infty$, we note that
\begin{align}\label{deltainfty}
\delta_\infty = e^{c_0}
\exp\bigg(\int_{\lambda^-}^{\lambda^+} \frac{s (x-L)}{\pi r_+(s)}ds\bigg)
\exp\bigg(\int_{\lambda^-}^{\lambda^+} \frac{2s^2 t}{\pi r_+(s)}ds\bigg),
\end{align}
where 
$$c_0 = -\frac{1}{2\pi i} \int_{\lambda^-}^{\lambda^+} \frac{\ln (- 2 i \lambda  r_+(s)) - \ln{q_0}}{r_+(s)}ds.$$
The integral in (\ref{deltainfty}) involving $x - L$  vanishes because the integrand is odd, and the integral involving $t$ can be computed by opening up the contour and deforming the contour to infinity:
$$\int_{\lambda^-}^{\lambda^+} \frac{2s^2 t}{\pi r_+(s)}ds
= -\frac{1}{2}\oint_{C} \frac{2s^2 t}{\pi r(s)}ds
= - i \lambda q_0^2 t,$$
where $C$ is a counterclockwise circle encircling the cut $[\lambda^-, \lambda^+]$.
If $\lambda = 1$, then the substitution $s = q_0 \sin\theta$ gives
\begin{subequations}\label{c0example}
\begin{align}
c_0 = \frac{1}{\pi} \int_{0}^{q_0} \frac{\ln (2\sqrt{q_0^2 - s^2}) - \ln{q_0}}{\sqrt{q_0^2 - s^2}}ds
= \frac{1}{\pi} \int_{0}^{\pi/2} \ln (2\cos{\theta})d\theta 
= 0,
\end{align}
while, if $\lambda = -1$, then the substitutions $s = i\sigma = i q_0 \sin\theta$ yield
\begin{align}\label{c0lambda1}
c_0 = \frac{1}{\pi} \int_{0}^{q_0} \frac{\ln (- 2 i \sqrt{q_0^2 - \sigma^2}) - \ln{q_0}}{\sqrt{q_0^2 - \sigma^2}}  d\sigma
= \frac{1}{\pi} \int_{0}^{\pi/2} \ln (- 2 i \cos\theta)d\theta
= -\frac{\pi i}{4}.
\end{align}
\end{subequations}
It follows that
\begin{align}\label{deltainftyfinal}
\delta_\infty = \begin{cases} e^{- i \lambda q_0^2 t}, & \lambda = 1, \\
e^{-\frac{\pi i}{4}} e^{- i \lambda q_0^2 t}, & \lambda = -1.\end{cases}
\end{align}
Substituting this expression for $\delta_\infty$  into (\ref{qfrommcheck}), we find that the solution $q(x,t)$ of (\ref{NLS}) corresponding to the constant initial data $q(x,0) = q_0 \in \R$ is given by 
\begin{align}\label{qexample}
q(x,t) = q_0 e^{-2i \lambda q_0^2 t}.
\end{align}
It is of course easily verified that this $q$ satisfies the correct initial value problem.

\bigskip
\noindent
{\bf Acknowledgements} {\it ASF acknowledges support from the EPSRC in the form of a senior fellowship. JL acknowledges support from the G\"oran Gustafsson Foundation, the Ruth and Nils-Erik Stenb\"ack Foundation, the Swedish Research Council, Grant No. 2015-05430, and the European Research Council, Grant Agreement No. 682537.}

\bigskip
\noindent
{\bf Conflict of interest}
On behalf of all authors, the corresponding author states that there is no conflict of interest.

\bibliographystyle{plain}
\bibliography{is}

\begin{thebibliography}{99}
\small

\bibitem{BK2011}
Computational approach to Riemann surfaces. Edited by Alexander I. Bobenko and Christian Klein. Lecture Notes in Mathematics, 2013. Springer, Heidelberg, 2011.

\bibitem{BBEIM1994}
E. D. Belokolos, A. I. Bobenko, V. Z. Enol'skii, A. R. Its, and V. B. Matveev, {\it Algebro-geometric approach to nonlinear integrable equations}, Springer series in nonlinear dynamics, Springer-Verlag, Berlin, Germany, 1994.

\bibitem{FKT2003}
J. Feldman, H. Kn\"orrer, and E. Trubowitz, {\it Riemann surfaces of infinite genus,} 
CRM Monograph Series, 20. American Mathematical Society, Providence, RI, 2003. 

\bibitem{F1997}
A. S. Fokas, A unified transform method for solving linear and certain nonlinear PDEs, 
{\it Proc. Roy. Soc. Lond. A} {\bf 453} (1997), 1411--1443.

\bibitem{FI1996}
A. S. Fokas and A. R. Its, The linearization of the initial-boundary value problem of the nonlinear Schr\"odinger equation,
{\it SIAM J. Math. Anal.} {\bf 27} (1996), 738--764. 

\bibitem{FI2004}
A. S. Fokas and  A. R. Its, The nonlinear Schr\"odinger equation on the interval, {\it J. Phys. A} {\bf 37} (2004), 6091--6114.

\bibitem{FLxperiodic}
A. S. Fokas and J. Lenells, A new approach to integrable evolution equations on the circle, {\it Proc. Roy. Soc. A} {\bf 477}: 20200605, \href{https://doi.org/10.1098/rspa.2020.0605}{https://doi.org/10.1098/rspa.2020.0605}.


\bibitem{GGKM1967}
C. S. Gardner, J. M. Greene, M. D. Kruskal, and R. M. Miura, Method for solving the Korteweg-de Vries equation, {\it Phys. Rev. Lett.} {\bf 19} (1967) 1095--1097.

\bibitem{GH2003}
F. Gesztesy and H. Holden, {\it Soliton equations and their algebro-geometric solutions. Vol. I. ($1+1$)-dimensional continuous models.} Cambridge Studies in Advanced Mathematics, 79. Cambridge University Press, Cambridge, 2003.


\bibitem{M2008}
V. B. Matveev, 30 years of finite-gap integration theory, {\it Philos. Trans. R. Soc. Lond. Ser. A Math. Phys. Eng. Sci.}  {\bf 366} (2008), 837--875. 

\bibitem{MvM1975}
H. P. McKean, and P. van Moerbeke, The spectrum of Hill's equation, {\it Invent. Math.} {\bf 30} (1975), 217--274.

\bibitem{MN2019}
K. T-R McLaughlin and P. V. Nabelek, A Riemann-Hilbert problem approach to infinite gap Hill's operators and the Korteweg-de Vries equation, {\it Int. Math. Res. Notices}, \href{https://doi.org/10.1093/imrn/rnz156}{https://doi.org/10.1093/imrn/rnz156}.

\bibitem{MSS1998}
W. M\"uller, M. Schmidt, and R. Schrader, Hyperelliptic Riemann surfaces of infinite genus and solutions of the KDV equation, {\it Duke Math. J.} {\bf 91} (1998), 315--352.


\end{thebibliography}

\end{document}